\documentclass[11pt]{article}

\usepackage[T1]{fontenc}
\usepackage[utf8]{inputenc}
\usepackage{lmodern}
\usepackage{microtype}

\usepackage{amsmath,amssymb,mathtools}
\usepackage{amsthm}
\usepackage{accents} 
\usepackage{mathrsfs}
\usepackage{enumerate}
\usepackage{comment}
\usepackage[ruled,lined,linesnumbered]{algorithm2e}

\usepackage{xcolor}
\usepackage{graphicx}
\usepackage{tikz}
\usepackage{subfigure}
\usepackage[affil-it]{authblk}
\usepackage[margin=1in]{geometry}

\usepackage{url}
\usepackage{hyperref}
\usepackage{thm-restate}
\usepackage[noabbrev,capitalise]{cleveref}

\providecommand{\lcm}{\operatorname{lcm}}

\theoremstyle{plain}
\newtheorem{thm}{Theorem}[section]

\newtheorem{lem}[thm]{Lemma}
\newtheorem{claim}{Claim}

{\noindent \emph{Proof.} {}{#1}{}}{\hfill
	$\Diamond$\vspace{1em}}

\theoremstyle{plain}

 \def\dfn#1{{\sl #1}}
 
\newcommand{\less}{\setminus}

\DeclareTextAccent{\caronaccent}{T1}{7}

\newcommand{\LL}{\mathcal L}
\DeclareMathOperator{\tw}{tw}

\usetikzlibrary{calc,positioning}

\title{Positively Lin--Lu--Yau curved graphs with no  $K_{2,t}$ minor}
\author{Shuliang Bai\thanks{Beijing Institute of Mathematical Sciences and Applications (BIMSA), Beijing 101408, China. E-mail address:baishuliang@bimsa.cn}, \quad Xiaonan Liu\thanks{Department of Mathematics, Louisiana State University, Baton Rouge, LA 70803, USA. Partially supported by the NSF-AWM Mentoring Travel Grant. E-mail address: xliu20@lsu.edu} \quad
 and \quad
  Zi-Xia Song\thanks{Department of Mathematics, University of Central Florida, Orlando, FL 32816, USA. Supported by the Simons Foundation through the Travel Support for Mathematicians program (SFI-MPS-TSM-00024413). E-mail address: Zixia.Song@ucf.edu}}

\date{}
\begin{document}
\maketitle
\begin{abstract}
 Let $G$ be a connected graph with minimum degree at least two, positive Lin-Lu-Yau Ricci curvature,  and  no $K_{2,t}$ minor, where  $t\ge 3$ is an integer.  We first establish tight upper bounds on the maximum degree of $G$ by showing that $\Delta(G)\le 2t+3$ for $t=4$,  $\Delta(G)\le 2t+4$ for $t\in\{5,6\}$, and $\Delta(G)\le 2t+2$ for all $t\ge7$.  We then  prove that $G$ has at most $10$ vertices for  $t=3$, and at most $(t+1)\left(\Delta(G)^2(\Delta(G)-1)/2+1\right)$ vertices for all  $t\ge 4$. 
 
 \end{abstract}

\baselineskip 18pt

\section{Introduction}

Ricci curvature plays an important role in the geometric analysis of Riemannian manifolds, linking local geometric information to global properties such as diameter, volume growth, and diffusion. Motivated by extending this local-to-global principle to discrete settings, several notions of discrete Ricci curvature have been introduced, including the Bakry--Émery curvature condition~\cite{Bakry-Emery1985} and Ollivier's coarse Ricci curvature on metric spaces~\cite{Ollivier}. In \cite{LLY}, Lin, Lu, and Yau modified Ollivier's Ricci curvature \cite{Ollivier} to introduce  a new variant of Ricci curvature on graphs, known as Lin--Lu--Yau Ricci curvature, which does not depend on the idleness of the random walk. 

All graphs considered in this paper are finite and simple. For a graph $G=(V(G), E(G))$, we   use $V(G)$ to denote the vertex set, $E(G)$ the edge set, $|G|$ the number of vertices, $e(G)$ the number of edges,   $\delta(G)$ the minimum degree, $\Delta(G)$ the maximum degree  and $\overline{G}$ the complement of $G$. For any positive integer $k$, we write  $[k]$ for the set $\{1,2, \ldots, k\}$. We use the convention   ``$A:=$'' to mean that $A$ is defined to be the right-hand side of the relation. 
For a graph $G$  and a vertex set $A\subseteq V(G)$, we define $N_G(A):=\{v\in V(G)\setminus A : va\in E(G) \text{ for some } a\in A\}$ and $N_G[A]:=N_G(A)\cup A$. 
When $A=\{a\}$, we simply write $N_G(a)$ and $N_G[a]$. Note that $d_G(a)=|N_G(a)|$. 
We use $d_G(x,y)$ to denote the distance between two vertices $x, y$ in $G$. More generally, for nonempty vertex sets $A,B\subseteq V(G)$, define $d_G(A,B):=\min\{d_G(a,b):a\in A,\ b\in B\}$, and let $e_G(A,B)$ denote the number of edges with one endpoint in $A$ and the other in $B$. When $A=\{a\}$, we write $d_G(a,B)$ and $e_G(a,B)$, respectively. If $G$ is clear from the context, we omit the subscript $G$. 

The \emph{union} $G\cup H$ of two graphs $G,H$ is the graph with vertex set $V(G)\cup V(H)$  and edge set $E(G)\cup E(H)$. The \emph{disjoint union} $G \sqcup H$ of $G, H$ is the union of vertex-disjoint copies of $G$ and $H$. For a positive integer $t$, we write $tG$ for the disjoint union of $t$ copies of $G$.
The {\dfn{join}} $G+H$ of two 
vertex disjoint graphs
$G$ and $H$ is the graph having vertex set $V(G)\cup V(H)$  and edge set $E(G)
\cup E(H)\cup \{xy\mid x\in V(G),  y\in V(H)\}$.

We use the same notation as in~\cite{LLY}.
A probability distribution (over the vertex set $V(G)$) is a mapping $m: V(G)\to [0,1]$ satisfying $\sum_{x\in V(G)} m(x) = 1$.
Suppose two probability distributions $m_1$ and $m_2$ have finite support. A {\em coupling} between $m_1$ and $m_2$ is a mapping $A: V(G)\times V(G) \to [0,1]$ with finite support so that 
$\sum_{y \in V(G)} A(x,y) = m_1(x) \textrm{ for every $x\in V(G)$ and } \sum_{x\in V(G)} A(x,y) = m_2(y)$ for every $y\in V(G)$. The \textit{Wasserstein-1 transportation distance} between two probability distributions $m_1$ and $m_2$ is defined as follows:
$$W(m_1, m_2) = \inf_A \sum_{x,y\in V(G)} A(x,y) d_G(x,y),$$
where the infimum is taken over all couplings $A$ between $m_1$ and $m_2$. A random walk $m$ on $G=(V,E)$ is defined as a family of probability measures $\{m_v(\cdot)\}_{v\in V(G)}$ such that $m_v(u) = 0$ for all $uv \notin E(G)$. It follows that  $m_v(u) \geq 0$ for all $v,u\in V(G)$ and $\sum_{u\in N[v]} m_v(u) = 1$ for all $v\in V(G)$. 

In~\cite{Ollivier}, Ollivier defined the \textit{coarse Ricci curvature} $\kappa: \binom{V(G)}{2} \to \mathbb{R}$ of a vertex pair in connected $G$ equipped with a random walk $m$ as follows: for distinct vertices $x, y\in V(G)$,  
$$\kappa_{\textrm{Ollivier}}(x,y) = 1 - \frac{W(m_x, m_y)}{d_G(x,y)}.$$
Note that $W(m_x, m_y)=W(m_y, m_x)$ for any distinct $x,y\in V(G)$, and hence $\kappa_{\textrm{Ollivier}}(x,y)=\kappa_{\textrm{Ollivier}}(y,x)$. One natural choice of random walks in graphs is the \textit{$\alpha$-lazy random walk} ($0\leq \alpha < 1$) defined as 
$$
m_x^{\alpha}(v) = \begin{cases} 
                        \alpha & \textrm{ if $v=x$,}\\
                        (1-\alpha)/d(x) &\textrm{ if $v\in N(x)$,}\\
                        0 & \textrm{ otherwise.}
                    \end{cases}
$$
In \cite{LLY}, Lin, Lu and Yau defined \emph{Lin--Lu--Yau Ricci curvature} (\emph{LLY curvature}, for short) of a vertex pair $\{x,y\}$ in connected $G$, denoted by $\kappa_{\textrm{LLY}}(x,y)$, based on the $\alpha$-lazy random walk in $G$ as $\alpha$ goes to $1$, i.e., $$\kappa_{\textrm{LLY}}(x,y) = \lim_{\alpha \to 1} \frac{\kappa_{\alpha}(x,y)}{(1-\alpha)},$$
where $\kappa_{\alpha}(x,y) = 1 - \frac{W(m_x^{\alpha}, m_y^{\alpha})}{d_G(x,y)}$. Similarly, we have that $\kappa_{\textrm{LLY}}(x,y)=\kappa_{\textrm{LLY}}(y,x)$ for any distinct $x,y\in V(G)$.
A connected graph $G$ is called \textit{positively Lin--Lu--Yau curved} or positively curved, for short, if $\kappa_{LLY}(x,y) >0$ for all $x \neq y \in V(G)$. Lin, Lu, and Yau~\cite{LLY} showed that if a connected graph has positive LLY curvature on every edge, then it has positive LLY curvature on every pair of distinct vertices. Consequently, a connected graph $G$ is positively curved if and only if every edge of $G$ has positive LLY curvature. Throughout this paper, all positively Lin--Lu--Yau curved graphs are assumed to be connected.

Let $U\subseteq V(G)$. A function $f:U\to\mathbb R$ is \emph{$1$-Lipschitz}, written $f\in Lip(1)$, if $|f(x)-f(y)|\le d_G(x,y)$ for all $x,y\in U$. M\"unch and Wojciechowski~\cite{MW} obtained a limit-free
characterization of Lin--Lu--Yau Ricci curvature in terms of the
normalized graph Laplacian $\LL$. Specifically, for any two
distinct vertices $x,y\in V(G)$, the LLY curvature of $\{x,y\}$ is the infimum of $\frac{\LL f(x) -\LL f(y)}{d(x,y)}$ taken over all $1$-Lipshchitz functions $f$ on $V(G)$ with $f(y)-f(x)=d(x,y)$. Moreover, this infimum is equivalently taken over all integer-valued $1$-Lipschitz functions on $N[x]\cup N[y]$ with $f(y)-f(x)=d_G(x,y)$.
\begin{thm}[M\"{u}nch and Wojciechowski~\cite{MW}]\label{thm:curvature_laplacian} Let $G= (V,E)$ be a connected graph and let $x \neq y \in V(G)$. Then 

\begin{equation*}
    \kappa_{\textrm{LLY}}(x,y) = \inf_{\substack{f: V(G) \to \mathbb{R}\\ f\in Lip(1)\\ \nabla_{yx}f = 1}} \nabla_{xy} \LL f  = \inf_{\substack{f: N[x]\cup N[y] \to \mathbb{Z} \\f \in Lip(1)\\ \nabla_{yx}f = 1}} \nabla_{xy} \LL f,
\end{equation*}
where $\nabla_{xy}f=\frac{f(x)-f(y)}{d(x,y)}$ and $\LL f(u)=\frac{1}{d(u)}\sum_{v\in N(u)} (f(v)-f(u))$. 
\end{thm}

Motivated by the study of planar graphs with positive combinatorial curvature, Lu and Wang~\cite{LuWangPlanarCurvature2023} initiated the study of the number of vertices  of positively Lin--Lu--Yau curved planar graphs. They proved that every positively curved planar graph with minimum degree at least three has bounded maximum degree and hence is finite. The minimum degree assumption is necessary, since $K_{2,t}$ is positively curved for every $t$. Brooks, et al.~\cite{BOSWY} characterized all positively curved maximal outerplanar graphs with minimum degree two. Building on their work, Liu, Lu and Wang \cite{LiuLuWang2024} subsequently characterized  all positively curved outerplanar graphs with minimum degree two and showed that every such graph has at most $10$ vertices. Again, the minimum degree condition is also necessary, since $K_{1,t}$ is positively curved for every $t$. We refer the reader to~\cite{GamlathLiuLuYuan2025,LinYou2026,ChenLinLiuYou2026,Hehl2026,BrooksJohnsonLinzLiuLuVaughn2026,RamosOlive2025,ChenHuLiuZhang2026,HuangLiuXia2024,LiLu2024}  for further results on this topic.

Motivated by these results on outerplanar graphs, we consider the broader class of graphs with no $K_{2,t}$ minor. We extend the result of Liu, Lu, and Wang by characterizing all positively curved nonouterplanar graphs with minimum degree two and no $K_{2,3}$ minor. For every $t\ge 4$, we first determine a sharp upper bound on the maximum degree of positively curved graphs with minimum degree at least two and no $K_{2,t}$ minor, and use this bound to show that, for each fixed $t$, every such graph is finite. 

\begin{thm}\label{t:maxdeg}
Let $t\ge4$ be an integer and let $G$ be a connected graph with $\delta(G)\ge2$,    positive Lin--Lu--Yau curvature, and no $K_{2,t}$ minor. Then
\[
\Delta(G)\le
\begin{cases}
2t+3 & \quad\text{ if } \quad t=4,\\
2t+4& \quad \text{ if }\quad t=5,6\\
2t+2 & \quad \text{ if }\quad t\ge7.
\end{cases}
\]
Each bound is sharp, even among graphs of minimum degree at least three.
\end{thm}

The bounds are sharp, as witnessed by the following examples: $K_1+C_{11}$ for $t=4$, $K_1+2F$ for $t=5$, $K_1+2C_8^2$ for $t=6$, and $K_1+2\overline{C_{t+1}}$ for $t\ge 7$, where $F$ is a $7$-vertex graph and $C_8^2$ denotes the square of $C_8$. See Figure~\ref{fig:F-C8-square}. It is also worth mentioning that the minimum degree two condition is necessary,  as   every star $K_{1,m}$ has positive Lin--Lu--Yau curvature.

\begin{figure}[htbp]
    \centering

    \resizebox{0.175\textwidth}{!}{
    \begin{tikzpicture}[
        scale=1.05,
        vertex/.style={circle, fill=black, inner sep=2.7pt}
    ]
        \node[vertex] (a) at (90:2) {};
        \node[vertex] (b) at (90:1) {};
        \node[vertex] (c) at (0:0) {};

        \node[vertex] (x) at (210:1) {};
        \node[vertex] (y) at ( -30:1) {};

        \node[vertex] (u) at (210:2) {};
        \node[vertex] (v) at (-30:2) {};

        \draw[thick] (a)--(u)--(v)--(a);

        \draw[thick] (a)--(b)--(c);
        \draw[thick] (c)--(x)--(u);
        \draw[thick] (c)--(y)--(v);

        \node at (0,-1.5) {$F$};
    \end{tikzpicture}}
    \hspace{1.5cm}
    \resizebox{0.175\textwidth}{!}{\begin{tikzpicture}[
        scale=1.05,
        vertex/.style={circle, fill=black, inner sep=2.7pt}
    ]
        \foreach \i in {0,...,7} {
            \node[vertex] (v\i) at ({90-45*\i}:2cm) {};
        }

        \draw[thick]
            (v0)--(v1)--(v2)--(v3)--(v4)--(v5)--(v6)--(v7)--(v0);

        \draw[thick]
            (v0)--(v2)
            (v1)--(v3)
            (v2)--(v4)
            (v3)--(v5)
            (v4)--(v6)
            (v5)--(v7)
            (v6)--(v0)
            (v7)--(v1);

        \node at (0,-2.55) {$C_8^2$};
    \end{tikzpicture}}

    \caption{The graphs $F$ and $C_8^2$.}
    \label{fig:F-C8-square}
\end{figure}
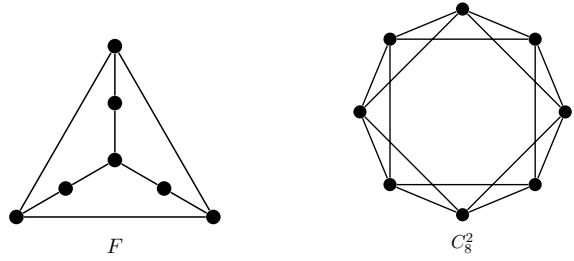

\begin{restatable}{thm}{vertexbound}\label{t:main}
Let $t\ge 3$ be an integer and   let $G$ be a connected graph with $\delta(G)\ge2$,    positive Lin--Lu--Yau curvature, and no $K_{2,t}$ minor. Then  
\[
 |G|\le
\begin{cases}
10& \quad \text{ if }\quad t=3,\\
(t+1)\left(\frac{\Delta(G)^2(\Delta(G)-1)}{2}+1\right) & \quad \text{ if } \quad t\ge 4. \end{cases}
\]
\end{restatable}

\medskip

We make no effort to improve the upper bounds in Theorem~\ref{t:main} for $t\ge4$.  
The remainder of the paper is organized as follows. Section~\ref{s:K_1t} presents several results on graphs with no $K_{1,t}$ minor. We prove Theorem~\ref{t:maxdeg} in Section~\ref{s:maxdeg}, Theorem~\ref{t:main} for $t=3$ in Section~\ref{s:t=3}, and Theorem~\ref{t:main} for all $t\ge4$ in Section~\ref{s:t=4}.

  \section{Structural properties of graphs with no \texorpdfstring{$K_{1,t}$}{K(1,t)} minor}\label{s:K_1t}
  
 Throughout this section, let  $t\ge 2$ be an integer.  We  present here  the    properties of graphs with no $K_{1,t}$ minor that  are used in the proofs of our main results.

\begin{lem}\label{lem:cone-minors_1}
 The graph $K_1+ H$ has no $K_{2,t}$ minor if and only if $H$ contains  neither   $K_{1,t}$  nor   $K_{2,t-1}$ as a  minor.
\end{lem}

We omit the proof of Lemma~\ref{lem:cone-minors_1}.  Given a graph $G$, a vertex set $A\subseteq V(G)$ is \dfn{connected} if $G[A]$ is connected.  A graph $G$ has a $K_{1,t}$ minor precisely when there exists a connected vertex set $A$ in $G$ such that $|N_G(A)|\ge t$.  In particular, if $G$ has no $K_{1,t}$ minor, then for each vertex $v\in V(G)$ and every non-empty connected set $A\subseteq V(G)$, we have 
\begin{equation}\label{eq:cone-boundary}
d_G(v)\le t-1  \text{ and } |N_G(A)|\le t-1.
\end{equation}

\begin{lem}\label{lem:pathcomp}
Let $G$ be a connected graph with  no $K_{1, t}$ minor. If $A\subseteq V(G)$ is a connected set  with $|N(A)|=t-1$, then 
\begin{enumerate}[(a)]
\item each vertex in $N_G(A)$ is adjacent to at most one vertex in $G\setminus N_G[A]$, and 
\item each component of  $G\setminus N_G[A]$ is a path such that no internal vertex on the path  is adjacent to any vertex in $N_G(A)$. 
\end{enumerate}
\end{lem}

\begin{proof} Let $G$ and $A$ be given as in the statement. Let $B:=V(G)\setminus N_G[A]$.  Suppose some vertex, say $x$, in $N_G(A)$ is adjacent to at least two vertices in $B$. Then $A\cup\{x\}$ is connected and $|N_G(A\cup\{x\})|\ge t$, contrary to  \eqref{eq:cone-boundary}. Thus each vertex in $N_G(A)$ is adjacent to at most one vertex in $B$.  Let $D$ be a component of $G[B]$. We may assume that $|D|\ge3$, else we are done. Since $G$ is connected, we see that $e_G(V(D),  N_G(A))>0$. Let  $y\in V(D)$  such that $y$ is adjacent to some vertex, say $x$, in $N_G(A)$. Then $d_{D}(y)=1$, else $A\cup\{x, y\}$ is connected and $|N_G(A\cup\{x,y\})|\ge t$, contrary to  \eqref{eq:cone-boundary}. Let $z$ be the unique neighbor of $y$ in $D$. Since $|D|\ge3$ and $D$ is connected, there exists $w\in V(D)$ with $w\ne y$ such that $zw\in E(G)$. Then   $z$ is adjacent to exactly two vertices $y, w$ in $D$ and $z$ is not adjacent to any vertex in $N_G(A)$. It follows that $D$ is a path  such that no   internal vertex on the path  is adjacent to any vertex in $N_G(A)$.  \end{proof}  
  
  \begin{lem}\label{lem:regular}
Let $t\ge3$, and let $G$ be a connected graph with $\delta(G)\geq t-1$. Then  $G=K_t$ or $G$ contains either   $K_{1,t}$   or  $K_{2,t-1}$ as a  minor.
 \end{lem}

\begin{proof} 
Let $G$ and $t$ be as in the statement. Suppose $G\ne K_t$ and $G$ has no $K_{1,t}$ minor. Then $\Delta(G)\leq t-1$ and so $G$ is $(t-1)$-regular because $\delta(G)\geq t-1$. Let $x\in V(G)$. For each $y\in N(x)$,    let $A:=\{x, y\}$. Then  $A$ is connected and $|N_G(A)|\le t-1$. Thus the edge $xy$ belongs to at least $t-3$ triangles in $G$.  It follows that   $\delta(G[N(x)])\ge t-3$. Let $B:=V(G)\less N[x]$. Then each vertex in $N(x)$ is adjacent to at most one vertex in $B$. Since $G\ne K_t$ is connected,   we see that  $B\ne\emptyset$ and  there exist $y\in N(x)$ and $z\in B$ such that $yz\in E(G)$. Note that $z$ is unique neighbor of $y$ in $B$.  By the previous argument, the edge $yz$ belongs to at least $t-3$ triangles, so $z$ is adjacent to at least $t-2$ vertices in $N(x)$. If $z$ is adjacent to all the vertices in $N(x)$, then   $B=\{z\}$ and $G$ contain $K_{2, t-1}$ as a subgraph. Thus there exists $w\in N(x)$ such that $wz\notin E(G)$. Then $w$ has no neighbor in $B$ and there exists  $z'\in B$ such that  $zz'\in E(G)$. But then the edge $zz'$ does not belong to a triangle in $G$, a contradiction.  
\end{proof}

\begin{lem}\label{lem:cone-minors}
 If  $G$ has no  $K_{1, t}$ minor, then no component of $G$ has a spanning tree with at least $t$  leaves.
\end{lem}
   
  \begin{proof}
Suppose some component of $G$ contains a spanning tree $T$ with at least $t$ leaves. Let $T'$ be obtained from $T$ deleting all the leaves in $T$. Now   contracting   $T'$ into a single vertex  yields  a $K_{1,t}$ minor in $G$, a contradiction.
  \end{proof}

 To apply  Lemma~\ref{lem:cone-minors} in our proof of Theorem
 `\ref{t:maxdeg}, we utilize  the following results on the maximum number of leaves of spanning trees in connected graphs. 
  
  \begin{thm}[Kleitman and West 1991 \cite{KW}]\label{t:mindeg=3}
If $G$ is a connected graph with $\delta(G)=3$, then $G$ has a spanning tree with at least $(|G|+8)/4$ leaves. 
\end{thm}

\begin{thm}[Kleitman and West 1991 \cite{KW}; Griggs and Wu 1992 \cite{GW}]\label{t:mindeg=4}
If $G$ is a connected graph with $\delta(G)=4$, then $G$ has a spanning tree with at least $(2|G|+8)/5$ leaves. 
\end{thm}

\begin{thm}[Griggs and Wu 1992 \cite{GW}]\label{t:mindeg5}
If $G$ is a connected graph with $\delta(G)\ge5$, then $G$ has a spanning tree with at least $(|G|+4)/2$ leaves. 
\end{thm}

\section{Proof of Theorem~\ref{t:maxdeg}}\label{s:maxdeg}
We begin this section with two key lemmas. For each vertex $v$ in a graph $G$, we define 
\[
b_G(v): =|\{u\in V(G)\mid d_G(v, u)=2\}|.
\]
Suppose  $G$ has no $K_{1,t}$ minor. Then for each vertex $v\in V(G)$ and every non-empty connected set $A\subseteq V(G)$, we have 
\[
d_G(v)\le t-1,\quad b_G(v)\le t-1 \quad \text{ and } |N_G(A)|\le t-1.
\]

\begin{lem}\label{lem:key}
Let $t\ge4$ be an integer and let $G$ be a connected graph with $\delta(G)\ge2$,    positive Lin--Lu--Yau curvature, and no $K_{2,t}$ minor.  If there exists a  vertex $u\in V(G)$  with $d_G(u) \ge 2t+3$, then there exists a  graph $H$ with $|H|=d_G(u) \ge 2t+3$ and $\delta(H)\ge1$ such that $K_1+ H$ is a minor of $G$ and every $v\in V(H)$ satisfies
\[\frac{|H|\cdot d_H(v)}{d_H(v)+1}<2d_H(v)+2+b_H(v). \]
\end{lem}

\begin{proof}
 Let $t, G, u$ be  as in the statement. 
 Fix $v\in N_G(u)$. Let 
\[
D:=d_G(u), \quad d:=d_G(v),\quad c: =|N_G(v)\cap N_G(u)|,\quad
A: =N_G[v]\setminus\{u\},\quad b:=|N_G(A)\cap N_G(u)|.
\]
Since $G$ has  no $K_{2,t}$ minor, we see that   $c\le t-1$ and  $b\le t-1$. Moreover, if $d=2$, then $b\le t-2$.  

For each vertex $x\in V(G)$, let  $f(x): =\min\{d_G(x,A),2\}$. Then  $f$ is a $1$-Lipschitz function with  $f(v)=0$, $f(u)=1$, $f(a)=0$ for each $a\in A$, $f(z)=1$ for each $z\in N_G(A)\cap N_G(u)$, and $f(w)=2$ for each $w\in N_G(u)\setminus N_G[A]$.  Consequently,
\[
\LL f(v)=\frac1d 
\text{ and } 
\LL f(u)=\frac{(c+1)(0-1)+b(1-1)+(D-b-c-1)(2-1)}D=1-\frac{2c+2+b}{D}.
\]
Recall that  $0<\kappa_{\textrm{LLY}}(v,u)\le \LL f(v)-\LL f(u)\le\frac1d-\big(1-\frac{2c+2+b}{D}\big)$ and $d\ge2$. It follows that 
\begin{equation}\label{eq:local}
\frac D2\le  D\left(1-\frac1d\right)<2c+2+b.
\end{equation}
Since $b\le t-1$ and $D\ge2t+3$, we see that  $c\ge1$. Moreover,  if    $c=1$, then $d=2$, $b=t-2$ and  $D=2t+3$. 
\medskip

Let $B:=\{w\in V(G)\mid d_G(u,w)=2\}$. For  each vertex $w\in B$, let $w'\in N_G(u)$ such that $ww'\in E(G)$.  Let $H$ be obtained from $G[N_G(u)\cup B]$ by contracting each of the edges $ww'$ and deleting all resulting loops and parallel edges. Then $|H|=|N_G(u)|=D\ge 2t+3$ and $K_1+H$ is a minor of $G$.   Let $d,c,A,b, B$ be    as defined above. Note that for each $v\in N_G(u)$, $r:=d_H(v)\ge c=|N_G(v)\cap N_G(u)|\geq 1$, and for each of the vertex $x$ in $N_G(v)\cap N_G(u)$ and $N_G(A)\cap N_G(u)$, we have $vx\in E(H)$  or $d_H(v, x)=2$.   Therefore, $\delta(H)\geq 1,$ $r+b_H(v)\ge c+b$ and so 
\begin{equation}\label{eq:transfer-boundary}
2r+2+b_H(v)=(r+b_H(v))+r+2\ge (c+b)+r+2=2c+2+b+(r-c).
\end{equation}
Moreover, if $r=1$, then  $c=1$ and     $D=2t+3$. We next show that 
\begin{equation}\label{eq:transfer-degree}
r-c+\frac{D}{r+1}\ge\frac Dd.
\end{equation}
Since $r\ge c$, we see that   inequality~\eqref{eq:transfer-degree} holds trivially if  $r+1\le d$. We may assume that $r\ge d$.  Then $v$ has at least one neighbor, say $v'$,  in $B$ such that contracting the edge $v''v'$ yields at least one new neighbor of $v$ in $N_H(v)\setminus A$, where $v''\in N_G(u)$. It follows that 
  $d\ge |N_G(v)\cap N_G(u)|+|\{u, v'\}|=c+2$. Then $c\le d-2$. Since $c\ge1$, we have $d\ge3$.  By~\eqref{eq:local}, 
  \[\frac{2t+3}2+D\left(\frac12-\frac1d\right)\le \frac D2+ D\left(\frac12-\frac1d\right)<2c+2+b\le 2(d-2)+2+(t-1),\] which implies that  
\[
D\left(\frac12-\frac1d\right)
<2d-\frac92,
\quad\text{so}\quad
D<\frac{d(4d-9)}{d-2}<4d\le d(d+1).
\]
Since $r\ge d$, we have  $D<d(d+1)\le d(r+1)$, and thus
\[
\frac Dd-\frac{D}{r+1}
=\frac{D(r+1-d)}{d(r+1)}
<r+1-d\le r+1-(c+2)<r-c.
\]
This proves~\eqref{eq:transfer-degree}.

Finally, combining~\eqref{eq:transfer-boundary}, \eqref{eq:transfer-degree}, and~\eqref{eq:local} in order yields
\begin{align*}
2r+2+b_H(v)+\frac{D}{r+1}
&\ge \bigl(2c+2+b+(r-c)\bigr)+\frac{D}{r+1}\\
&= 2c+2+b+\left(r-c+\frac{D}{r+1}\right)\\
&\ge (2c+2+b)+\frac{D}{d}\\
&> D\left(1-\frac{1}{d}\right)+\frac{D}{d}\\
&= D.
\end{align*}
Thus $Dr/(r+1)=D-D/{r+1}<2r+2+b_H(v)$ for every $v\in V(H)$, as desired.  \medskip

This completes the proof of Lemma~\ref{lem:key}.
\end{proof}
\begin{lem}\label{lem:H}

Let $H$ be a graph with $\delta(H)\geq 1$ and for each $v\in V(H)$,
   \begin{equation}\label{ineq}
\frac{|H|\cdot d_H(v)}{d_H(v)+1}<2d_H(v)+2+b_H(v). 
\end{equation}
Suppose
  \begin{equation}\label{ineq:H_maxdeg}
|H| \ge
\begin{cases}
2t+4 & \text{ for}\quad t=4,\\
2t+5&  \text{ for}\quad t=5,6\\
2t+3 &  \text{ for}\quad t\ge7,
\end{cases}
\end{equation}
where $t\ge4$ is an integer. 
Then $H$ contains  either  $K_{1,t}$  or $K_{2,t-1}$ as a minor.
\end{lem}
\begin{proof}
Let $H, t$ be as in the statement. Suppose $H$ contains neither  $K_{1,t}$ nor $K_{2,t-1}$ as a minor. Then  $\Delta(H)\leq t-1$.

\begin{claim}\label{cl:deg_1} 
If $d_H(v)=1$ for some $v\in V(H)$, then $b_H(v)=t-2$ and $|H|=2t+3$.
\end{claim}
\begin{proof}
Suppose $v\in V(H)$ has $d_H(v)=1$, and let $u$ be the unqiue neighbor of $v$. Note that $|b_H(v)|=|N(u)\setminus \{v\}|=d_H(u)-1\leq \Delta(H)-1\leq t-2$. By inequality~\eqref{ineq},
\[\frac{|H|}{2}<2+2+b_H(v)\leq t+2,\]
and so $|H|<2t+4$. Then inequality~\ref{ineq:H_maxdeg} implies that $|H|=2t+3$ and $b_H(v)=t-2$.
\end{proof}
\begin{claim}\label{cl:deg_2} 
If $d_H(v)=2$ for some $v\in V(H)$, then $b_H(v)=t-1$ and $t\in \{4, 7,8\}$. Moreover, $|H|\in \{2t+4, 2t+5\}$ when $t=4$, and $|H|=2t+3$ when $t\in \{7,8\}$.
\end{claim}
\begin{proof}
Suppose $v\in V(H)$ has $d_H(v)=2$. Note that $|b_H(v)|\leq t-1$. By inequality~\eqref{ineq},
\[\frac{2|H|}{3}<4+2+b_H(v)\leq t+5,\]
and so \[|H|<\frac{3t+15}{2}=2t+3-\frac{t-9}{2}=2t+4-\frac{t-7}{2}=2t+5-\frac{t-5}{2}.\] Then inequality~\ref{ineq:H_maxdeg} implies that $t\in \{4, 7,8\}$. Moreover,  $|H|\in \{2t+4, 2t+5\}$ when $t=4$, and $|H|=2t+3$ when $t\in \{7,8\}$. Note that $b_H(v)>2|H|/3-6$. Hence, if $t=4$, then $b_H(v)>2(2t+4)/3-6= 2=t-2$, and if $t\in \{7,8\}$, then $b_H(v)>2(2t+3)/3-6>t-2$. In both cases, $b_H(v)=t-1$ as $b_H(v)\leq t-1$.
\end{proof}
\begin{claim}\label{cl:special_A} Let $C$ be a component of $H$. If there exists a connected set  $A\subseteq V(C)$ such that  $|N_C(A)|=t-1$, then  
\[|C|-|N_C[A]|\le |N_C(A)|=t-1.\]  
 \end{claim}
 
 \begin{proof} Let $J:=C\setminus N_C[A]$ and let $B$ be a component of $J$. Since $C$ has no $K_{1, t}$ minor and $|N_C(A)|=t-1$,  by Lemma~\ref{lem:pathcomp}, each vertex in $N_C(A)$ is adjacent to at most one vertex in $J$  and   $B$ is a path such that no   internal vertex on the path  is adjacent to any vertex in $N_C(A)$. Thus $e_C(N_C(A), V(J))\le |N_C(A)|$.  Let $v_1, \ldots, v_\ell$ be the vertices of $B$ in order, where $\ell=|B|$. If $\ell\ge5$, then $d_C(v_3)=d_H(v_3)=2$ and   $b_H(v_3)=|\{v_1, v_5\}|=2<t-1 $, contrary to   Claim~\ref{cl:deg_2}.  Thus  $1\le \ell \le 4$.
 We next show that $e_C(V(B), N_C(A))\geq |B|$. Since $C$ is connected, we have  $e_C(V(B),N_C(A))\geq 1$.  We may assume that $2\le \ell\le4$.  Suppose $d_C(v_1)=1$. By Claim~\ref{cl:deg_1},    $b_H(v_1)=t-2\geq 2$. Then $\ell=2$ and so $e_C(V(B), N_C(A))=b_H(v_1)\ge |B|$. We may assume that 
 $d_C(v_1)\geq 2$ and $d_C(v_\ell)\geq 2$.  Then $e_C(V(B), N_C(A))\geq 2$. 
We may further assume   that $3\le \ell\le4$.  Then $d_C(v_2)=d_H(v_2)=2$ and $d_C(v_{\ell-1})=d_H(v_{\ell-1})=2$.   By Claim~\ref{cl:deg_2}, $b_H(v_2)=t-1\geq 3$ and $b_H(v_{\ell-1})=t-1\geq 3$. It follows  that $e_C(V(B), N_C(A))= b_H(v_2)\ge3=|B|$ when $\ell=3$, and  $e_C(V(B), N_C(A))= (b_H(v_2)-1)+ (b_H(v_{\ell-1})-1)\ge 4 =|B|$ when $\ell=4$. By the arbitrary choice of $B$, this proves that   $e_C(V(B), N_C(A))\geq |B|$ for every component $B$ of $J$. Therefore,  
 \[|C|-|N_C[A]|=|J|\leq e_C(V(J), N_C(A))=e_C(N_C(A), V(J))\leq |N_C(A)|=t-1,\] as desired.
\end{proof}

\begin{claim}\label{cl:C<=G/2}
 Suppose $C$ is a component of $H$ with $|C|\le |H|/2$.  Then $\delta(C)=1$ and $|C|=t+1$. Moreover, $d_H(v)=1$ or $d_H(v)=t-1$ for each $v\in V(C)$.
\end{claim}
\begin{proof}
Let $v\in V(C)$, and let $r:=d_C(v)=d_H(v)$. We claim that $r\neq 2$. Suppose $r=d_C(v)=2$. Then $d_H(v)=d_C(v)=2$. It follows from Claim~\ref{cl:deg_2} that $b_H(v)=t-1$ and $t\in \{4, 7,8\}$. Moreover, $|H|\in \{2t+4, 2t+5\}$ when $t=4$, and $|H|=2t+3$ when $t\in \{7,8\}$. Observe that $|C|\geq 1+r+b_H(v)=1+2+t-1=t+2$.
 Since $t+2\leq |C|\leq |H|/2$, we have $t=4$ and $|H|\in \{2t+4, 2t+5\}$, which  implies that $|C|=6$. By Claim~\ref{cl:deg_1},  $\delta(H)\neq 1$ since   $|H|\neq 2t+3$. Thus $2\le \delta(H)\leq \delta (C)\leq \Delta(C)\leq \Delta(H)\leq t-1=3$.  Let $u_1, u_2$ be the two neighbors of $v$ in $C$ and let $\{w_1, w_2, w_3\}:=V(C)\setminus \{v, u_1, u_2\}$. Recall that $b_H(v)=t-1=3$. Then $N_H(\{u_1, u_2\})|= \{v, w_1, w_2, w_3\}$, and so $u_1u_2\notin E(H)$, else $\{u_1, u_2\}$ is a connected set in $H$ with $|N_H(\{u_1, u_2\})|=4$ and so H contains a $K_{1,4}$ minor, a contradiction.  We may assume that $u_1w_1, u_1w_2\in E(H)$. Then $d_C(u_1)=3$ and $u_2w_3\in E(H)$.  Since $H$ contains no $K_{2,3}$ minor, we see that $w_3$ is not adjacent to both $w_1$ and $w_2$, We may assume that $w_1w_3\notin E(H)$. Since $\delta(H)\ge2$, we see that  $w_3w_2\in E(H)$. Moreover, either    $w_1u_2\in E(H)$ or $w_1w_2\in E(H)$. In the first case we obtain a $K_{2,3}$ minor, and in the latter case $d_H(u_2)=2$ with $b_H(u_2)=|\{u_1, w_2\}|=2<t-1$, contrary to Claim~\ref{cl:deg_2}. By the arbitrary choice of $v$, this proves that no vertex in $C$ has   degree two.

    Suppose next $r\neq 1$.  
Since $|C|\le |H|/2$ and $r+b_H(v)\le |C|-1$, by inequality~\eqref{ineq}, 
  \[\frac{|H|\cdot r}{r+1}< (r+b_H(v))+r+2 \le  (|C|-1)+r+2 \le \frac{|H|}2+r+1,\]
  which implies that 
\begin{equation}\label{ineq:C}
|C|\leq \frac{|H|}2<\frac{(r+1)^2}{r-1}=r+3+\frac 4{r-1} \, \text{ and } \,|C|>\frac{|H|\cdot r}{r+1}-1-r.
\end{equation}

We claim that $r=t-1$. Suppose $r\le t-2$. Then $3\le r\le t-2$ and so $t\ge5$.   Thus $|H|\ge15$ by inequality~\eqref{ineq:H_maxdeg}.
 Suppose  $r=3$. Then the inequalities in~\eqref{ineq:C} imply that $|C|<3+3+2=8$ and  $|C|>(15\cdot 3)/(3+1)-1-3=7+1/4$, which is impossible. Suppose next $r=4$. Then  $t\ge r+2\ge 6$ and $|H|\ge 17$. Moreover, the inequalities in~\eqref{ineq:C} imply that $|C|<4+3+4/3$ and  $|C|>(17\cdot 4)/(4+1)-1-4=8+3/5$, which again is impossible. Thus $r\ge 5$,  and so $t\ge r+2\ge 7$ and $|H|\ge 2t+3\ge 2(r+2)+3=2r+7$.  
 By the inequalities in~\eqref{ineq:C} again, we have 
 \[|C|<r+3+4/(r-1)\le r+4 \, \text{ and }\, |C|>(2r+7)r/(r+1)-1-r=r+4-5/(r+1),\]
 which is a contradiction. This proves that every vertex in $C$ has degree one or degree $t-1$. By Claim~\ref{cl:deg_1}, $\Delta(C)=t-1$. 

 It remains to show that $\delta(C)=1$ and $|C|=t+1$. Let $u\in V(C)$ with $d_H(u)=d_C(u)=t-1$. By applying the second inequality in~\eqref{ineq:C} to the vertex $u$, we have 
 \[|C|> \frac{|H|\cdot(t-1)}{t} -1-(t-1)\geq \frac{(2t+3)(t-1)}t-t=t+1-\frac{3}{t}.\]
Thus $|C|\geq t+1$ because $t\ge4$. Suppose $\delta(C)\ne1$. Then $C$ is $(t-1)$-regular. By Lemma~\ref{lem:regular}, $C$ contains either $K_{1,t}$ or $K_{2, t-1}$ as a minor, a contradiction. This proves that $\delta(C)=1$. By Claim~\ref{cl:deg_1},  $|H|=2t+3$, and so  $t+1\le |C|\le (2t+3)/2$. Thus     $|C|=t+1$, as desired.
  \end{proof}

  \begin{claim}\label{cl:Hcon} $H$ is connected. 
  \end{claim}
  
  \begin{proof}  
  Suppose $H$ is disconnected. Let $C$ be a component of $H$ such that  $|C|\leq \frac{|H|}{2}$. By Claim~\ref{cl:C<=G/2}, $\delta(C)=1$ and $|C|=t+1\ge5$, and for every $v\in V(C)$, either $d_C(v)=1$ or $d_C(v)=t-1$. Let $v\in V(C)$ with $d_C(v)=\delta(C)=1$. Let $w$ denote the unique neighbor of $v$. By Claim~\ref{cl:deg_1},  $d_C(w)=t-1$. Since $|C|=t+1$, let $x$ denote the unique vertex in $V(C)\setminus N_C[w]$. Observe that $N_C(x)\subseteq N_C[w]\setminus \{v, w\}$. Thus  $d_C(x)\leq t-2$ and so $d_C(x)=1$. Let $y\in N_C[w]\setminus \{v, w\}$ be the unique neighbor of $x$. Then   $d_C(y)=t-1$ by Claim~\ref{cl:deg_1}. This implies that $y$ is adjacent to every vertex in $N_C(w)\setminus\{v,y\}$. Consequently, every $z\in N(w)\setminus\{v,y\}$ is adjacent to both $w$ and $y$. Moreover, $z$ has at most   $\bigl|V(C)\setminus\{z,v,x\}\bigr|=t-2 $ neighbors in $C$,   contrary to the fact that $d_C(z)=t-1$. 
  \end{proof}

  By Claim~\ref{cl:Hcon}, $H$ is connected.   Suppose $1\leq \delta(H)\leq 2$. By Claim~\ref{cl:deg_1}  and Claim~\ref{cl:deg_2}, there exists a connected set $A\subseteq V(H)$ with $|A|\leq 3$ such that $|N_H(A)|\le t-1$. Then $|N_H[A]|=|A|+|N_H(A)|\leq t+2$. By Claim~\ref{cl:special_A}, $|H|-|N_H[A]|\leq t-1$ and so  $|H|\leq 2t+1$, a contradiction. This proves that $\delta(H)\geq 3$.
  By Lemma~\ref{lem:cone-minors}, $H$ contains no spanning tree with at least $t$ leaves. By Theorem~\ref{t:mindeg5}, $3\le \delta(H)\le 4$. Let  $v\in V(H)$ with $d_H(v)=\delta(H)$. Applying  the inequality~\eqref{ineq} to the vertex  $v$, we see that $4\le t\le 9$ when $\delta(H)=3$,  and $4\le t\le 10$ when $\delta(H)=4$. Suppose $\delta(H)=4$.  By Theorem~\ref{t:mindeg=4}, $H$ contains a spanning tree with at least $(2|H|+8)/5\ge (4t+14)/5\ge t$ leaves because $t\le 10$, contrary to Lemma~\ref{lem:cone-minors}.  Thus  $\delta(H)=3$.   If $b_H(v)=t-1$, by Claim~\ref{cl:special_A} applied to $N_H[v]$, $|H|-4-(t-1)\le t-1$, contrary to the fact that $|H|\ge 2t+3$.  Thus $b_H(v)\le t-2$. By  the  inequality~\eqref{ineq} applied to the vertex $v$, we have $4\le t\le7$.  
By Theorem~\ref{t:mindeg=3} and the inequality~\eqref{ineq:H_maxdeg} applied to $|H|$, $H$ contains a spanning tree with at least $\lceil (|H|+8)/4\rceil\ge   t$ leaves, contrary to Lemma~\ref{lem:cone-minors}. 

This completes the proof of Lemma~\ref{lem:H}.
\end{proof}

 We now apply Lemmas~\ref{lem:key} and~\ref{lem:H} to prove Theorem~\ref{t:maxdeg}.
 
\begin{proof}[Proof of Theorem~\ref{t:maxdeg}]  Let $t\ge 4$ be an integer and let $G$ be a connected positively Lin--Lu--Yau curved graph with $\delta(G)\ge2$ and  no $K_{2,t}$ minor. Suppose 
\begin{equation*}
\Delta(G)\ge
\begin{cases}
2t+4 & \text{ if } t=4,\\
2t+5&  \text{ if } t=5,6\\
2t+3 &  \text{ if }t\ge7.
\end{cases}
\end{equation*}
By Lemma~\ref{lem:key}, there exists a  graph $H$ with $|H|=\Delta(G) \ge 2t+3$ and $\delta(H)\ge1$ such that $K_1+ H$ is a minor of $G$ and every $v\in V(H)$ satisfies
\[\frac{|H|\cdot d_H(v)}{d_H(v)+1}<2d_H(v)+2+b_H(v). \]
Since 
\begin{equation*}
|H|=\Delta(G)\ge
\begin{cases}
2t+4 & \text{ if } t=4\\
2t+5&  \text{ if } t=5,6\\
2t+3 &  \text{ if }t\ge7,
\end{cases}
\end{equation*}
it follows from Lemma~\ref{lem:H} that $H$ contains $K_{1,t}$ or $K_{2, t-1}$ as a minor. Thus  $K_1+H$ contains a $K_{2,t}$ minor and so  $G$ contains a $K_{2,t}$  minor,  a contradiction.
\end{proof}

\section{Proof of Theorem~\ref{t:main} for \texorpdfstring{$t=3$}{t = 3}}\label{s:t=3}
For $t=3$, we give a characterization of positively Lin--Lu--Yau curved non-outerplanar graphs with minimum degree at least two and no $K_{2,3}$ minor. Recall that a graph is \dfn{outerplanar} if it has neither $K_{2,3}$ nor $K_4$ as a minor. Liu, Lu, and Wang~\cite{LiuLuWang2024} characterized all positively Lin--Lu--Yau curved outerplanar graphs with minimum degree at least two. 
\begin{thm}[Liu, Lu, Wang \cite{LiuLuWang2024}]
 Every connected positively Lin--Lu--Yau curved outerplanar graph $G$ with $\delta(G)\ge 2$ has $\Delta(G)\leq 9$ and $|G|\le 10.$   Both bounds are sharp.
\end{thm}

Moreover,  every graph containing a $K_4$ minor either contains a subgraph isomorphic to $K_4$ or contains a $K_{2,3}$ minor. This observation leads to the following lemma. 

\begin{lem}\label{Lem:K23-minor}
Every $2$-connected graph with no $K_{2,3}$ minor is either outerplanar or isomorphic to $K_4$.    
\end{lem}

 Let $G\ne K_4$ be a connected positively Lin-Lu-Yau curved graph with $\delta(G)\geq 2$ and no $K_{2,3}$ minor. By Lemma~\ref{Lem:K23-minor}, either $G$ is outerplanar or $G$ contains a cut vertex. Let $B$ be a block of $G$. We call a vertex of $B$ \emph{active} if it is a cut vertex of $G$. By Lemma~\ref{Lem:K23-minor}, $G$ contains a block isomorphic to $K_4$ if and only if  $G$ is not outerplanar. 

 For an integer $n\ge1$, let $P_n$ denote the path on $n$ vertices. The \emph{fan graph} $F_n$ is the join $K_1+P_n$. The vertex corresponding to $K_1$, which is adjacent to every vertex of $P_n$, is called the \emph{center} of $F_n$.
\begin{lem}\label{lem:K4_block}
 Let $G$ be a connected positively Lin--Lu--Yau curved graph with $\delta(G)\ge 2$ and no $K_{2,3}$ minor, and suppose that $G$ contains a block $B$ isomorphic to $K_4$.  Then the following hold.
 \begin{enumerate}[(i)]
 \item  At most one vertex of $B$ is active.
 \item   Suppose that $v$ is the unique active vertex of $B$ incident with a block $C$ of $G$, where $C\ne B$. Then $d_G(r)\leq 8$ and either $C\cong K_4$  or $C\cong K_1+P_{\ell}$ with $2\le \ell\le 5$ is a fan graph with $v$ as its center. Moreover, $v$ is the unique active vertex of $C$.  
 \end{enumerate}
\end{lem}
\begin{proof} Let $G,B$ be given in the statement. Suppose $v,a$ are two distinct active vertices of $B$.  Note that $va\in E(B)$ as $B\cong K_4$. Since $G$ is positively Lin--Lu--Yau curved with $\delta(G)\geq 2$, we know that $G$ cannot have a cut edge, and so all blocks of $G$ are $2$-connected. Then each of $v,a$ have at least two neighbors outside of $B$, and hence $d_G(v)\geq 5,  d_G(a)\geq 5$. Observe that for any vertex  $u\in N_G(v)\setminus V(B)$ and any vertex $w\in N_G(a)\setminus V(B)$, the distance between $u$ and $w$ is three. Otherwise, this $uw$-path of length at most two contains no vertex in $V(B)\setminus \{a, v\}$ and contracting this path together with $B$ gives a $K_{2,3}$ minor.  

Consider the edge $av$, and let $f: N[a]\cup N[v] \to \mathbb{Z}$ be defined as follows. 
\[f|_{V(B)\setminus\{ v\}}=0, \quad f(v)=1,  \quad f|_{N_G(a)\setminus V(B)}=-1, \quad f|_{N_G(v)\setminus V(B)}=2.\]
It follows that $f\in Lip(1)$ with $f(a)=0, f(v)=1$ and
\[\LL f(a)=\frac{-(d_G(a)-3)+1}{d_G(a)}=\frac{4}{d_G(a)}-1, \quad \LL f(v)=\frac{-3+(d_G(v)-3)}{d_G(v)}=1-\frac{6}{d_G(v)},\]
and $\kappa_{\textrm{LLY}}(a,v)\leq \LL f(a)-\LL f(v)=\frac{4}{d_G(a)}+\frac{6}{d_G(v)}-2\leq \frac{4}{5}+\frac{6}{5}-2=0,$ a contradiction. Hence, $B$ has at most one active vertex. This proves (i).

Suppose $v$ is the unique active vertex of $B$. We first prove that $v$ has degree at most $8$. Let $a\in V(B)$ with $a\neq v$. By (a), we know that $a$ is not active, and so $d_G(a)=d_B(a)=3$. Consider the edge $va$, and $1$-Lipschitz function $f: N[a]\cup N[v] \to \mathbb{Z}$ with 
\[f|_{V(B)\setminus\{ v\}}=0, \quad f(v)=1,  \quad f|_{N_G(v)\setminus V(B)}=2.\]
Then
\[\LL f(a)=\frac{1}{3}, \quad \LL f(v)=\frac{-3+(d_G(v)-3)}{d_G(v)}=1-\frac{6}{d_G(v)}.\]
$0<\kappa_{\textrm{LLY}}(a,v)\leq \LL f(a) -\LL f(v)=\frac{1}{3}-1+\frac{6}{d_G(v)}$ implies that $d_G(v)\le 8$.

Suppose $v$ is incident with a block $C$ of $G$ such that $C\neq B$. Suppose $C\cong K_4$. Then by (i), $v$ is the unique active vertex of $C$. Assume $C$ is not $K_4$. Then $C$ is outerplanar. We want to show that $C$ is a fan graph with center $v$ and $v$ is the unique active vertex of $C$. We claim that if $x$ is a neighbor of $v$ in $C$, then $N_G(x)\subseteq N_C(v)$. We first show that $x$ is not active in $C$. Suppose not. Then $x$ has at least two neighbors outside of $C$. Consider the edge $vx$, and let $f: N[v]\cup N[x] \to \mathbb{Z}$ be the function with 
\[f|_{\{v\} \cup  (N_C(v)\setminus N_C[x])}=0, \quad f|_{N_C[x]\setminus \{v\}}=1,  \quad f|_{ V(B)\setminus \{v\}}=-1, \quad f|_{N_G(x)\setminus V(C)}=2.\]
Since $C$ is outerplanar, $v$ and $x$ have at most two common neighbors in $C$.
Observe that 
\[\LL f(r)=\frac{-3+1+|N_C(v)\cap N_C(x)|}{d_G(v)}\leq 0, \quad \LL f(x)=\frac{-1+|N_G(x)\setminus V(C)|}{d_G(x)}>0,\]
and $\kappa_{\textrm{LLY}}(v,x)\le \LL f(v)-\LL f(x)<0$, contradicting that $\kappa_{\textrm{LLY}}(v,x)>0$. Therefore, $x$ is not active in $C$, and so $N_G(x)=N_C(x)$. Now we show that $N_C(x)\subseteq N_C(v)$. Suppose $x$ has a neighbor $y$ in $C$ such that $vy\notin E(G)$. Note that $x\in N_C(y)\cap N_C(v)$, and since $C$ is outerplanar, we have $|N_C(v)\cap N_C(x)|+|N_C(v)\cap N_C(y)\setminus N_C(x)|\leq 3$. Let $f: N[v]\cup N[x] \to \mathbb{Z}$ be a $1$-Lipschitz function such that 
\[f|_{\{v\} \cup  (N_C(v)\setminus (N_C(x)\cup N_C(y)))}=0, \quad f|_{(N_C(x)\setminus \{v, y\})\cup (N_C(v) \cap N_C(y))}=1,  \quad f|_{ V(B)\setminus \{v\}}=-1, \quad f(y)=2.\]
Then
\[\LL f(v)\le \frac{-3+|N_C(v)\cap N_C(x)|+|N_C(v)\cap N_C(y)\setminus N_C(x)|}{d_G(v)}\leq 0, \quad \LL f(x)=\frac{-1+1}{d_G(x)}=0,\]
contradicting that $\kappa_{\textrm{LLY}}(v,x)>0$. Thus, $N_G(x)=N_C(x)\subseteq N_C(v)$. Since $C$ is outerplanar, we obtain that $C$ is a fan graph $K_1+P_\ell$ for some integer $\ell$ with center $v$, and $v$ is the unique active vertex of $C$ since $N_C[v]=V(C)$ and all neighbors of $v$ in $C$ cannot be active in $C$. Note that since $d_G(v)\le 8$ and $d_C(v)\geq 2$, we obtain that $2\le d_C(v)=\ell\leq d_G(v)-d_B(v)\leq 8-3=5$ and hence $2\le \ell \le 5$.
This completes the proof of Lemma~\ref{lem:K4_block}.
\end{proof}
\begin{thm}\label{thm:K_23}
Let $G$ be a connected positively Lin-Lu-Yau curved graph with $\delta(G)\geq 2$ and no $K_{2,3}$ minor. Suppose $G$ is not outerplanar. Then $G$ is isomorphic to one of the following graphs: \[
\begin{gathered}
K_4,\qquad K_1+2K_3,\qquad K_1+(K_3\sqcup 2P_2), \qquad
K_1+(K_3\sqcup P_\ell)\quad (2\le \ell\le5).
\end{gathered}
\] In particular, $\Delta(G)\leq 8$ and $|G|\leq 9$. Both bounds are attained by $K_1+(K_3\sqcup P_5)$.
\end{thm}
\begin{proof}
Let $G$ be connected positively Lin-Lu-Yau curved graph with $\delta(G)\geq 2$ and no $K_{2,3}$ minor such that $G$ is not outerplanar. Suppose $G$ is $2$-connected. It follows from Lemma~\ref{Lem:K23-minor} that $G\cong K_4$. We may now assume $G$ is not $2$-connected. Note that $G$ is $2$-edge connected as $G$ is positively Lin-Lu-Yau curved. 
Consider all blocks of $G$. Since $\delta(G)\geq 2$ and $G$ is $2$-edge connected, every block of $G$ is $2$-connected. Moreover, for any distinct $x,y$ in a block $B$ of $G$, there exist exactly two internally vertex-disjoint $xy$-paths in $G$, i.e., $G\setminus (V(B)\setminus \{x,y\})$ has no $xy$-path. Otherwise, it gives a $K_{2,3}$ minor of $G$. Since $G$ is not outerplanar, this implies $G$ has a block $B$ such that $B$ is not outerplanar and so $B\cong K_4\cong K_1+K_3$. Note that $B$ has at least one active vertex as $G$ is not $2$-connected. By Lemma~\ref{lem:K4_block}(i), $B$ has a unique active vertex, say $v$. Let $C\neq B$ be a block of $G$ such that $v\in V(C)$. It follows from   Lemma~\ref{lem:K4_block}(ii) that $d_G(v)\leq 8$, and either $C\cong K_4$  or $C\cong K_1+P_{\ell}$ with $2\le \ell\le 5$ is a fan graph with $v$ as its center. Moreover, $v$ is the unique active vertex of $C$. 

Suppose $d_C(v)=2$. This implies that $C\cong K_1+P_2$ is a triangle. We show that $d_G(v)\leq 7$. Let $x\in V(C)$ with $x\neq v$. Consider the edge $vx$ and let $f: N[v]\cup N[x]\to \mathbb{Z}$ be a $1$-Lipschitz function with
\[f(v)=0, \quad f|_{V(C)\setminus \{v\}}=1, \quad f|_{N_G(v)\setminus V(C)}=-1. \]
Then
\[\LL f(v)=\frac{-(d_G(v)-2)+2}{d_G(v)}=\frac{4}{d_G(v)}-1, \quad \LL f(x)=-\frac{1}{2},\]
and $0<\kappa_{\textrm{LLY}}(v,x )\leq \LL f(v) -\LL f(x)=\frac{4}{d_G(v)}-\frac{1}{2}$ gives that $d_G(v)\leq 7$. Hence, if $v$ is incident with one more block $D$ with $D\notin\{ B,C\}$, then $v$ has degree exactly two in $D$ and so $D\cong K_1+P_2\cong K_3$. Thus, $G\cong K_1 +(K_3\sqcup P_2)$ or $G\cong K_1 +(K_3\sqcup 2P_2)$. We may now assume that $d_D(v)\geq 3$ for any block $D$ containing $v$. 

It follows that $G$ has exactly two blocks $B\cong K_4=K_1+K_3$ and $C$ as $d_G(v)\leq 8$. Suppose $C$ is fan graph with center $v$. Then $C\cong K_1+P_\ell$ with $2\le \ell\le 5$. Observe that $d_C(r)=\ell\geq 3$ and hence $G\cong K_1+(K_3\sqcup P_\ell)$ with $3\le \ell \leq 5$. We may now assume that $C\cong K_4$. Then $G\cong K_1+2K_3$. Therefore,  $\Delta(G)\leq 8$ and $|G|\leq 9$. Both bounds are attained by $K_1+(K_3\sqcup P_5)$.

This completes the proof of Theorem~\ref{thm:K_23}.
\end{proof}

\section{Proof of Theorem~\ref{t:main} for \texorpdfstring{$t\ge4$}{t >= 4}}\label{s:t=4}

 To prove Theorem~\ref{t:main} for $t\ge4$, we use the following results concerning the second-smallest eigenvalues of the normalized Laplacian and the Laplacian of a connected graph. For a graph $G$, let $A$ and $D$ denote its adjacency matrix and diagonal degree matrix, respectively, and define \[ \mathscr L:=I-D^{-1/2}AD^{-1/2} \qquad\text{and}\qquad L:=D-A. \] Let $\mu_2$ and $\lambda_2$ denote the second-smallest eigenvalues of $\mathscr L$ and $L$, respectively.


\begin{thm}[Lin, Lu, and Yau~\cite{LLY}]\label{t:curvature_to_spectral}
Let $G$ be a connected graph and $\mu_2$ be the second smallest eigenvalue of $\mathscr L=I-D^{-1/2}AD^{-1/2}$. Then 
\[\mu_2\geq \min_{xy\in E(G)}\kappa_{\textrm{LLY}}(x,y).\]
\end{thm}

\begin{thm}[Ostrowski~\cite{Ostrowski1959}]\label{t:spectral}
Let $G$ be a connected graph with minimum degree $\delta$. Let $\mu_2$ and $\lambda_2$ be the respective second-smallest eigenvalues of $\mathscr L=I-D^{-1/2}AD^{-1/2}, L=D-A$.  Then 
\[\lambda_2\geq \delta \mu_2.\]
\end{thm}
\begin{thm}[Gima, Hanaka, Noro, Ono, Otachi~{\cite{GHNOO}}]\label{thm:spectral-treewidth}
Every $n$-vertex graph $G$ with maximum degree $\Delta(G)\ge1$ has treewidth
\begin{equation}\label{eq:spectral-treewidth}
\tw(G)\ge\frac{n\lambda_2}{\Delta(G)+\lambda_2}-1,
\end{equation}
where $\lambda_2$ is the second-smallest eigenvalue of $L=D-A$.
\end{thm}

\begin{thm}[Liu, Yoo~{\cite{LY}}]\label{thm:apex-forest}
If a graph $F$ becomes a forest after deleting at most one vertex, then every graph of treewidth at least $|V(F)|-1$ contains $F$ as a minor.
\end{thm}
We  are now ready to prove Theorem~\ref{t:main} for $t\ge 4$.
\begin{proof}[Proof of Theorem~\ref{t:main} for $t\ge 4$]
Let $t\ge 4$ be an integer, and let $G$ be a connected positively Lin--Lu--Yau curved graph with $\delta(G)\geq 2$ and no $K_{2,t}$ minor. It follows from Theorem~\ref{t:maxdeg} that $\Delta(G)$ is bounded, and so, for any edge $xy\in E(G)$, the infimum in the equality for $\kappa_{\textrm{LLY}}(x,y)$ in Theorem~\ref{thm:curvature_laplacian} can be replaced by minimum. Since  $\kappa_{\textrm{LLY}}(x,y)>0$, we have that $\kappa_{\textrm{LLY}}(x,y)\geq \frac{1}{\lcm(d(x), d(y))}\geq \frac{1}{\Delta(G)(\Delta(G)-1)}$ for any $xy\in E(G)$. Let $\lambda_2$ be the second-smallest eigenvalue of $L=D-A$, where $A$ and $D$ denote the adjacency matrix and diagonal degree matrix of $G$, respectively. Thus, it follows from Theorem~\ref{t:curvature_to_spectral} and Theorem~\ref{t:spectral} that $\lambda_2\geq \frac{\delta(G)}{\Delta(G)(\Delta(G)-1)}\geq \frac{2}{\Delta(G)(\Delta(G)-1)}.$ Since $K_{2,t}$ becomes a forest after deleting one  vertex of degree $t$, every graph with no $K_{2,t}$ minor has treewidth at most $|V(K_{2,t})|-2=t$ by Theorem~\ref{thm:apex-forest}. Hence, $G$ has treewidth at most $t$. By Theorem~\ref{thm:spectral-treewidth}, 
\[|G|\leq (t+1)\left(\frac{\Delta(G)}{\lambda_2}+1\right)\leq (t+1) (\frac{\Delta(G)^2(\Delta(G)-1)}2+1).\]
This completes the proof of Theorem~\ref{t:main}.
\end{proof}

\noindent{\bf Declaration of  Generative AI Use}: {The authors used ChatGPT 5.6 to assist in formulating the statements of the main results and to generate preliminary proof attempts. The proofs presented in this paper were subsequently developed by the authors, primarily using their own mathematical ideas and reasoning, and were written in their own words. All final arguments and proofs were independently verified by the authors. The authors take full responsibility for the accuracy, integrity, and final content of the manuscript.}

\bibliographystyle{abbrv}
\bibliography{reference}

\end{document}